\documentclass[a4paper,fleqn]{cas-sc}
\usepackage{amsthm}
\usepackage{subcaption}
\usepackage{float}

\renewcommand{\printorcid}{} 
\usepackage{tikz,everypage}
\newcommand{\tr}{ {\rm tr} }
\newcommand{\T}{\top}
\newcommand{\R}{\mathbb R}
\newcommand{\Exp}[1]{E_{\theta,\Sigma}^{#1}}

\newtheorem{theorem}{Theorem}

\newtheorem{lemma}{Lemma}
\newtheorem{corollary}{Corollary}
\makeatletter
\@addtoreset{theorem}{section}
\@addtoreset{equation}{section}
\@addtoreset{lemma}{section}
\@addtoreset{corollary}{section}
\@addtoreset{proposition}{section}
\makeatother

\usepackage[authoryear,longnamesfirst]{natbib}

\begin{document}
\let\WriteBookmarks\relax
\def\floatpagepagefraction{1}
\def\textpagefraction{.001}
\shorttitle{}

\title [mode = title]{Covariance matrix estimation under data--based loss}                      



\author[1]{Dominique Fourdrinier}

\fnmark[1]
\ead{Dominique.Fourdrinier@univ-rouen.fr}
\credit{Conceptualization, Methodology, Supervision, Validation, Writing - review \& editing, Writing - original draft, Software}
\address[1]{Universit\'e de Normandie, UNIROUEN, UNIHAVRE, INSA Rouen, LITIS, avenue de l'Universit\'e, BP 12, 76801 Saint-\'Etienne-du-Rouvray, France.}
\author[2]{Anis M. Haddouche}
\cormark[1]
\fnmark[2]
\ead{Mohamed.haddouche@insa-rouen.fr}
\address[2]{INSA Rouen, LITIS and LMI, avenue de l'Universit\'e, BP 12, 76801 Saint-\'Etienne-du-Rouvray, France.}
\credit{Conceptualization, Methodology, Supervision, Validation, Writing - review \& editing, Writing - original draft, Software}
\author[3]{Fatiha Mezoued}
\fnmark[1]
\ead{famezoued@yahoo.fr}
\address[3]{\'Ecole Nationale Sup\'erieure de Statistique et d\textquoteright
	\'Economie Appliqu\'ee (ENSSEA), LAMOPS, Tipaza, Algeria.}
\credit{Conceptualization, Methodology, Supervision, Validation, Writing - review \& editing, Writing - original draft, Software}

\cortext[cor1]{Corresponding author}
\fntext[fn1]{Professor}
\fntext[fn2]{Temporarily associated to teaching and research.}

\begin{abstract}
In this paper, we consider the problem of estimating the $p\times p$ scale matrix $\Sigma$ of a multivariate linear regression model $Y=X\,\beta + \mathcal{E}\,$ when the distribution of the observed matrix $Y$ belongs to a large class of elliptically symmetric distributions.
After deriving the canonical form $(Z^{\T} U^{\T})^{\T}$ of this model, any estimator $\hat{ \Sigma}$ of $\Sigma$ is assessed through the data--based loss 
$ \tr\,(S^{+}\Sigma\, (\Sigma^{-1}\hat{\Sigma} - I_p)^2	)\,$
where  $S=U^{\T}U$ is the sample covariance matrix and $S^{+}$ is its Moore-Penrose inverse.
We provide alternative estimators to the usual estimators $a\,S$, where $a$ is a positive constant, which present smaller associated risk.
Compared to the usual quadratic loss $\tr (\Sigma^{-1}\hat{\Sigma} - I_p)^2$, we obtain a larger class of estimators and a wider class of elliptical distributions for which such an improvement occurs. A numerical study illustrates the theory.	
\end{abstract}


\maketitle


\section{Introduction}\label{Introduction}

Let consider the multivariate linear regression model,  with $p$ responses and $n$ observations,
\begin{align}\label{linear.model}
	Y=X\,\beta + \mathcal{E}\,,
\end{align}
where $Y$ is an $n\times p$ matrix, $X$ is an $n\times q$ matrix of known constants  of rank $q\leq n$ and $\beta$ is a $q\times p$ matrix of unknown parameters. We assume that the $n\times p$ noise matrix ${\mathcal{E}}$ has an elliptically symmetric distribution with density, with respect to the Lebesgue measure in $\R^{pn}$,  of the form
\begin{align}\label{noise.density}
	{\varepsilon} \mapsto |\Sigma|^{-n/2} \,f  \big( \tr ( \,{\varepsilon}\,\Sigma^{-1}{\varepsilon}^{\T })\big)\,,
\end{align}
where $\Sigma$ is a $p\times p$ unknown positive definite matrix and $f(\cdot)$ is a non--negative unknown function.

The model \eqref{linear.model} has been considered by various authors such as  \cite{kubokawa1999,Kubokawa2001}, who estimated $\Sigma$ and $\beta$ respectively in the context \eqref{noise.density}, and \cite{tsukuma2016unified} who estimated $\Sigma$ in the Gaussian setting. A common alternative representation of this model is $Y=M+\mathcal{E}$, where $\mathcal{E}$ is as above and $M$ is in the column space of $X$, has been also considered in the literature. See for instance  \cite{CanuFourdrinier2017} and  \cite{Candesetal2013}.

Although the matrix of regression coefficients $\beta$ is also unknown, we are interested in estimating the scale matrix $\Sigma$. We address this problem under a decision--theoretic framework  through a canonical form of the model \eqref{linear.model}, which
allows to use a sufficient statistic $S=U^{\T}\,U$ for $\Sigma$, where $U$ is an $(n-q) \times p$ matrix (see Section \ref{main.results} for more details). In this context, the  natural estimators   of $\Sigma$  are of the form 
\begin{align}\label{natural.estimators}
	\hat{\Sigma}_a=a\,S\,,
\end{align}
for some positive constants $a$.
 
As pointed out by  \cite{James1961a}, the estimators of the form \eqref{natural.estimators} perform poorly in the Gaussian setting. In fact, larger (smaller) eigenvalues of $\Sigma$ are overestimated (underestimated) by those estimators. 
Thus we may expect to improve these estimators by shrinking the eigenvalues of $S$, which gives rise to the class of orthogonaly invariant estimators (see \cite{Takemura1984}). Since  the seminal work of  \cite{James1961a}, this problem has been largely considered in the Gaussian setting. See, for instance,  \cite{tsukuma2016unified},  \cite{tsukuma2016} and \cite{ChetelatWells2016}. However, the elliptical setting has been considered by a few authors such as \cite{kubokawa1999},  \cite{HADDOUCHE2021104680}. 

In this paper,
the performance of any estimator $\hat{\Sigma}$ of $\Sigma$ is assessed through the data-based loss
\begin{align}\label{loss}
	L_S(\hat{\Sigma},\Sigma)  
	= \tr\,\big(S^{+}\Sigma\, \big(\Sigma^{-1}\hat{\Sigma} - I_p\big)^2	\big)\,
\end{align}
and its associated risk
\begin{align}\label{risk}
	R(\hat{\Sigma},\Sigma)
	=\Exp{}\big[ \tr\,\big(S^{+}\Sigma\, \big(\Sigma^{-1}\hat{\Sigma} - I_p\big)^2	\big)	\big]\,,
\end{align}
where  $\Exp{}$ denotes the expectation  with respect to the density specified below in \eqref{density} and where $S^{+}$ is the Moore--Penrose inverse of $S$.
Note that, when  $p > n-q$, $S$ is non--invertible and, 
 when  $p \leq n-q$, $S$ is invertible so that $S^{+}$ coincides with the regular inverse $S^{-1}$. 
This type of loss is called data--based loss in so far as it contains a part of the observation $U$ through $S=U^{\T}\,U$. The notion of  data--based loss was introduced by  \cite{EfronMorris1976} when estimating a location parameter. Likewise, \cite{FourdrinierStrawderman2015} showed the interest of considering such a data--based loss with respect to the usual quadratic losses. Also, the data--based loss \eqref{loss} was considered, in a Gaussian setting, by \cite{TsukumaKubokawa2015} who were motivated by the difficulty to handle with the standard quadratic loss  
\begin{align}\label{quadratic.loss}
	L(\hat{\Sigma},\Sigma)=\tr \big(\Sigma^{-1}\hat{\Sigma} - I_p\big)^2\,.
\end{align}
%
See \cite{Haff1980} and  \cite{tsukuma2016} for more details.
Thus the loss in \eqref{loss} is a data--based variant of the \eqref{quadratic.loss}, through which we aim to improve on the estimators $\hat{\Sigma}_{a}$ in \eqref{natural.estimators} by alternative estimators, focusing on improved  orthogonally invariant estimators.
Note that most improvement results in the Gaussian case were derived thanks to Stein--Haff types identities. Here, we specifically use the Stein--Haff type identity given by \cite{HADDOUCHE2021104680}, in the elliptical case, to establish our dominance result, which is well adapted to our unified approach of the cases $S$ invertible and $S$ non--invertible.

The rest of this paper is structured as follows. In Section \ref{main.results}, we give  improvement conditions of the proposed estimators over the usual estimators. In Section \ref{numerical.study}, we assess the quality of the proposed estimators through a simulation study in the context of the t--distribution. We also compare numerically our results with those of \cite{Konno2009} in the Gaussian setting. Finally, we give in an Appendix  all the proofs of our findings.

%
\section{Main results}\label{main.results}
Although we are interested in estimating the scale matrix $\Sigma$, recall that $\beta$ is a $q\times p$ matrix of unknown parameters. Note that, since $X$ has full column rank, the least square  estimator of $\beta$ is
$
\hat{\beta}=(X^{\top}X)^{-1}\,X^{\top}Y
;$
this is the maximum likelihood estimator in the Gaussian setting.
Natural estimators of the scale matrix $\Sigma$ are based on the residual sum of squares given by 
\begin{align}\label{S.1}
	S=Y^{\top}\,(I_n - P_{X})\,Y,
\end{align}
where $P_X=X\,(X^{\top}X)^{-1}\,X^{\top}$  is the orthogonal projector onto the subspace spanned by the columns of $X$. 

Following the lines of  \cite{kubokawa1999} and \cite{Tsukuma2020},
we derive the canonical form of the model \eqref{linear.model} which allows a suitable treatment of the estimation of $\Sigma$. Let $X=Q_1\,T^{\T}$ be the $QR$ decomposition of $X$ where $Q_1$ is a $n\times q$ semi-orthogonal matrix and $T$ a $q\times q$ lower triangular matrix with positive diagonal elements. Setting  $m=n-q$, there exists a $n\times m$ semi-orthogonal matrix $Q_2$ which completes $Q_1$ such that $Q=(Q_1 Q_2)$ is an $n\times n$ orthogonal matrix. Then, since 
\newpage
\begin{align*}
	Q_2^{\T}\,X\,\beta = Q_2^{\T}\,Q_1\,T^{\T}\,\beta= 0\
\end{align*}
we have 
\begin{align}\label{canonique.mode}
	Q^{\T}\,{Y}
	=
	\begin{pmatrix} 
		{Z} \\ {U}  
	\end{pmatrix}
	=
	\begin{pmatrix} 
		{Q_{1}^{\T}} \vspace{.1cm}\\  {Q_2^{\T}}  
	\end{pmatrix}
	\,X\,\beta + Q^{\T}{\mathcal{E}}
	=
	\begin{pmatrix}
		{\theta}\\ {0}
	\end{pmatrix}
	+Q^{\T}{\mathcal{E}}\,,
\end{align}
where $	Q_1^{\T}\,X\,\beta =\theta\,$ and where $Z$ and $U$ are, respectively,  $q\times p$ and  $m\times p$ matrices. As $X=Q_1\,L^{\top}$, the projection matrix $P_X$ satisfies $P_X= Q_1\,L^{\top}(L^{\top}\,L)^{-1}L\,Q_1^{\top}=Q_1\,Q_1^{\top}$ so that  $I_n-P_X= Q_2\,Q_2^{\top}$. It follows that \eqref{S.1} becomes  $$S   =Y^{\T}Q_2\,Q_2^{\T}Y=U^{\T}\,U,$$  according to \eqref{canonique.mode}, which is a sufficient statistic for $\Sigma$.

The orthogonal matrix $Q$ provides a linear reduction from $n$ to $q$ observations within each of the $p$ responses. In addition, according to \eqref{noise.density}, the density of $Q^{\T}\mathcal{E}$ is the same as that of $\mathcal{E}$, and hence, $(Z^{\T} U^{\T})^{\T}$ has an elliptically symmetric distribution about the matrix $(\theta^{\T} 0^{\T})^{\T}$ with density 
\begin{align}\label{density}
	({z},{u}) 
	&\mapsto
	|\Sigma|^{-n/2} \,f \big(\,\tr\, ({z} 
	- {\theta})\,\Sigma^{-1}\,({z}-
	{\theta})^{\T } + \tr\, {u}\,\Sigma^{-1}\,{u}^{\!\T }
	\,\big)\,,
\end{align}
where $\theta$ and $\Sigma$ are unknown. In this sense, the model \eqref{canonique.mode} is the canonical form of the multivariate linear regression model \eqref{linear.model}.
Note that the marginal distribution of $U= Q_2^{\T}\,Y$ is elliptically symmetric about $0$ with covariance matrix proportional to $I_m \otimes \Sigma$ (see \cite{FangZhang1990}).
This implies that $S=U^{\top}\,U$  have a generalized Wishart distribution (see \cite{Diaz2011}), which coincides with the standard (singular or non--singular) Wishart distribution in the Gaussian setting (see \cite{Srivastava2003}).

As mentioned in Section \ref{Introduction}, the usual estimators of $\hat{ \Sigma}_a$ in \eqref{natural.estimators} perform poorly.
We propose alternative estimators of the form 
\begin{align}\label{alternative.estimators.0}
	\hat{\Sigma}_{J}=a\,(S + J)\,,
\end{align}
where $J=J(Z,S)$ is a correction matrix.
The improvement over the class of estimators $\hat{ \Sigma}_a$ can be done by improving  the best estimator $\hat{\Sigma}_{a_{o}}=a_o\,S$ within this class, namely, the estimator which minimizes the risk \eqref{risk}. It is proved in the Appendix  that
\begin{align}\label{optimal.estimator}
	\hat{\Sigma}_{a_{o}}= a_{o}\,S\,,
	\quad
	\text{with}
	\quad
	a_{o}=\frac{1}{K^{*}\,v}\quad \text{and}\quad v=\max\{p,m\}\,,
\end{align}
%
where $K^{*}$ is the normalizing constant (assumed to be finite) of the density defined by 
\begin{align}\label{density.F*}
	({z},{u}) \mapsto \frac{1}{K^{*}}|\Sigma|^{-n/2}\,F^{*}
	\big(\,\tr\,({z} - 	
	{\theta})\,\Sigma^{-1}\,({z}-{\theta})^{\top} + \tr\,u\,\Sigma^{-1}\,u^{\T}\,\big)\,,
\end{align} 
where, for any  $t\geq 0$, $${F^{*}(t) =\frac{1}{2} \, \int^{\infty}_{t} f(\nu) \,d\nu}\,.$$
Note that under de quadratic loss function \eqref{quadratic.loss} the optimal constant is $1/K^{*}(p+m+1)$.
Of course, this risk optimality has sense only  if  the risk of $\hat{\Sigma}_{a_o}$ is finite. As shown in  \cite{haddouche:tel-02376077}, this is the case as soon as $\Exp{}\big[\tr\big(\Sigma^{-1}S\big)\big]<\infty$ and $\Exp{}\left[\tr\big(\Sigma\, S^{+}\big)\right]<\infty$.

In order to give a unified dominance result of $\hat{\Sigma}_{J}$ over 
$\hat{\Sigma}_{a_{o}}$ for the two cases where $S$ is  non--invertible  and where $S$ is invertible, we consider, as a correction matrix in \eqref{alternative.estimators.0},  the projection of a matrix function $G(Z,S)=G$ on the subspace spanned by the columns of $SS^{+}$, namely, 
\begin{align}\label{alternative.estimators.1}
	J=SS^{+}G\,.
\end{align}
In addition to the risk finiteness  conditions of $\hat{ \Sigma}_{a_o}$, it can be shown that the risk of $\hat{\Sigma}_{J}$ is finite as soon as the expectations $\Exp{}\big[\|\Sigma^{-1}SS^{+}G\|_F^2\big]$ and $\Exp{}\big[\|S^{+}G\|_F^2\big] $ are finite, where $\|\cdot \|_F$ denotes the Frobenius norm. Under these conditions, the risk difference between  $\hat{\Sigma}_{J}$ and $\hat{ \Sigma}_{a_o}$ is
\begin{align}\label{Delta.SG.1}
	\Delta(G)
	&=
	a^{2}_{o}\Exp{}\big[ \,\tr\big(\Sigma^{-1}\,SS^{+}\,G\{I_p+ S^{+}G + SS^{+}	\}\big)\big]
	-2\,a_{o}\,\Exp{}\big[\tr\big(S^{+}\,G\big)\big] \,.
\end{align}
Noticing that the first integrand term in \eqref{Delta.SG.1} depends on the unknown parameter $\Sigma^{-1}$, our approach consists in replacing this integrand term by a random matrix $\delta(G)$, which does not depend on $\Sigma^{-1}$, such that $\Delta(G)\leq \Exp{*}\big[\delta(G)\big]$ where $\Exp{*}$ denotes the expectation with respect to the density \eqref{density.F*}. Clearly, a sufficient condition for $\Delta(G)$ to be non--positive (and hence, for  $\hat{\Sigma}_{J}$ to improve over $\hat{\Sigma}_{a_{o}}$) is that $\delta(G)$ is non--positive. To this end, we rely on  the following Stein--Haff type identity.
\begin{lemma}[\cite{HADDOUCHE2021104680}]\label{lemma.id.hs}
	Let $G(z,s)$ be a $p \times p$ matrix function such that, for any fixed $z$, $G(z,s)$ is weakly differentiable with respect to $s$. Assume that $\Exp{} \big[| \tr (\Sigma^{-1} S \,S^{+}\,G) | \big] < \infty$. Then we have 
	\begin{align}\label{Stein-Haff.id}
		\Exp{}\big[\tr\big(\,\Sigma^{-1}\,SS^{+}\,G\big)\big] 
		=
		K^{*}\,\Exp{*}\big[\tr \big(
		2\,SS^{+}\,{\cal D}_s
		\{SS^{+}G\}^{\T}\, + (m-r-1)\,S^{+}\,G\big)\,\big]\,,
	\end{align}
	where $r=\min\{p,m\}$  and ${\cal D}_s \{\cdot\}$  is the Haff operator whose generic element is
	%
	$ \frac{1}{2}\,(1 + \delta_{ij})\,\frac{\partial}{\partial S_{ij}},  
	$
	with $\delta_{ij}=1$ if $i=j$ and $\delta_{ij}=0$ if $i\neq j$.
\end{lemma}

Note that the existence of the expectations in \eqref{Stein-Haff.id} is implied by the above risk finiteness conditions.
An original Stein--Haff identity was derived independently by  \cite{Stein1977} and  \cite{haff1979} in the Gaussian setting where $S$ is invertible. This identity was extended to the class of elliptically symmetric distributions in \eqref{density}  \cite{kubokawa1999} and also by   \cite{GuptaBodnar2009}. Here, we use the new Stein--Haff type identity recently derived  by  \cite{HADDOUCHE2021104680} in the elliptical framework \eqref{density}  dealing with both cases $S$ non--invertible and $S$ invertible.

Applying Lemma \ref{lemma.id.hs} to the term depending on $\Sigma^{-1}$ in the right--hand side of \eqref{Delta.SG.1} gives 
\begin{align}\label{Delta.SG.2.S}
	\Delta(G)
	& =
	a_o^2\,K^{*}\,\Exp{*}\big[(m-r -1 )\, \tr\big(S^{+}G + (S^{+}G)^2  + S^{+}GSS^{+} \big) 
	%
		 \nonumber	\\
		&\hspace{2cm}
	%
	+2\,\tr \big(SS^{+}\,{\cal D}_s\{SS^{+}G + SS^{+}GS^{+}G +SS^{+}\,G\,SS^{+} \}^{\T}	\big)\big]  
	-2\,a_o\,\Exp{}\big[\tr\big(S^{+}\,G\big)\big].
\end{align}

It is worth noticing that the risk difference in \eqref{Delta.SG.2.S} depends on  the $\Exp{}$ and $\Exp{*}$ expectations (which coincide in the Gaussian setting since $F^{*}=f$). Thus, in order to derive a dominance result, we need to compare these two expectations. A possible approach consists to restrict us to the  subclass of densities verifying $c \leq{F^{*}(t)}/{f(t)}\leq b$, for some positive constants $c$ and $b$ (see  \cite{Berger1975} for the class where $c \leq{F^{*}(t)}/{f(t)}$). Due to the complexity of the use of the quadratic loss in \eqref{quadratic.loss} (which necessitates a twice application of the Stein--Haff type identity \eqref{Stein-Haff.id}), this subclass was considered by  \cite{HADDOUCHE2021104680}.
Here, thanks to the data--based loss \eqref{loss}, we are able to avoid such a  restriction, and hence, to deal with a larger class of elliptically symmetric distributions in \eqref{density}  (subject to the moment conditions induced by the above finiteness conditions). 

Following the suggestion to shrink the eigenvalues of $S$  mentioned in Section \ref{Introduction}, we consider as a correction matrix a matrix  $SS^{+}G$ with $G$ orthogonally invariant in the following sense.
Let $S=H\,L\,H^{\top}$ the eigenvalue decomposition of $S$
where $H$ is a $p\times r $ semi--orthogonal matrix of eigenvectors and $L={\rm diag}(l_1,\dots,l_r)$, with $l_1 >,\dots,>l_r$, is the diagonal matrix of the $r$ positive corresponding eigenvalues of $S$ (see  \cite{KubokawaSrivastava2008} for more details).
Then set $G=H\,L\Psi(L)\,H^{\top}$, with $\Psi(L)={\rm diag}(\psi_1(L),\dots,\psi_r(L))$ where $\psi_i=\psi_i(L)$ ($i=1,\dots,r$) is a differentiable function of $L$. Consequently, by semi--orthogonality of $H$, we have $SS^{+}H=H\,H^{\top}H=H$, so that the correction matrix in \eqref{alternative.estimators.1} is
$$J=SS^{+}G=G=H\,L\Psi(L)\,H^{\top}.$$
Thus the alternative estimators that we consider are of the form 
\begin{align}\label{alternative.estimators}
	\hat{\Sigma}_{\Psi}
	&= a_{o}\,\big(S + H\,L\,\Psi(L)\,H^{\T}) = a_{o}\,H\,L\,\big(I_{r}+\Psi(L)\big)\,H^{\T}\,,
\end{align}
which are usually called orthogonally invariant estimators
(i.e. equivariant under orthogonal transformations). See for instance  \cite{Takemura1984}.


Now, adapting the risk finiteness conditions mentioned above, we are in a position to give our dominance result of the alternative estimators in \eqref{alternative.estimators} over the optimal estimator  in  \eqref{optimal.estimator}, under the data--based loss \eqref{loss}.
\begin{theorem}\label{risk.diff.}
	Assume that the following expectations $\Exp{}\big[\tr(\Sigma^{-1}S)\big]$, $\Exp{}\big[\tr(\Sigma^{}S^{+})\big]$, $\Exp{}\big[\|\Sigma^{-1}HL\Psi(L) H^{\T} \|_F^{2} \big]$ and  $\Exp{}\big[\|H\Psi(L) H^{\T} \|_F^{2} \big]$ are finite.
	Let $\Psi(L)={\rm diag}(\psi_1,\dots,\psi_r)$ where $\psi_i=\psi_i(L)$ ($i=1,\dots,r$) is differentiable function of $L$ with $\tr\big(\Psi(L)\big)\geq \lambda$, for a fixed positive constant $\lambda.$
	
	Then an upper bound of the risk difference between $\hat{\Sigma}_\Psi$  and $\hat{\Sigma}_{a_o}$ under the loss function \eqref{loss} is given by 
	\begin{align*}
		\Delta(\Psi(L)) \leq	
		a_{o}^{2}\,K^{*}\,\Exp{*}\big[ g (\Psi)\big]\,,
	\end{align*}
	where 
	\begin{align}\label{tool.t12.3}
		g(\Psi)=
		\sum_{i=1}^{r} \left\{ 2(v-r+1)\psi_i
		+ (v-r+1)\psi_i^2 
		+4l_i(1+\psi_i )\frac{\partial \psi_i}{\partial l_i}
		+\sum_{j\neq i}^{r}\frac{l_i\,(2\psi_i + \psi_i^2)-l_j(2\psi_j + \psi_i^2)}{l_i -l_j}
		-2v\lambda \right\}.
	\end{align} 
	Also $\hat{\Sigma}_{\Psi}$ in \eqref{alternative.estimators} improves over $\hat{\Sigma}_{a_o}$  in  \eqref{optimal.estimator} as soon as $g(\Psi)\leq0$.
\end{theorem}
The proof of Theorem \ref{risk.diff.} is given in the Appendix.
Note that, although the expectation $\Exp{*}$ is associated to the generating function $f(\cdot)$ in \eqref{noise.density}, the function $g(\Psi)$ does not depend on $f(\cdot)$, and hence, the improvement result in Theorem \ref{risk.diff.} is robust in that sense.
Note also that Theorem \ref{risk.diff.} is well adapted to deal with the  \cite{James1961a} estimator where $\psi_i(L)=1/(v+r-2i+1)$, for $i=1,\dots,r$, since $\tr\big(\Psi(L)\big)>\lambda=1/(v+r-1)$ and the  Efron-Morris-Dey estimator, considered by  \cite{tsukuma2020estimation}, where $\psi_i(L)=1/\big(1+b\,l_i^\alpha/\tr(L^{\alpha})\big)v$, for $i=1,\dots,r$ and for positive constants $b$ and $\alpha$, since $\tr\big(\Psi(L)\big)>\lambda=r\,/(b+1)\,v$. 

In the following, we  consider a new  class of estimators which is an extension of the  \cite{Haff1980} class,
that is,  estimators of the form
\begin{align}\label{haff.estimators}
	\hat{\Sigma}_{\alpha,b} = a_{o}\,\big(S + H\,L\,\Psi(L)\,H^{\T}\big) 
	%
	\,\, \text{with, for}\,\, \alpha \geq 1 \,\, \text{and}  \,\, b >0, 
	%
	\,\,\Psi(L) = b\,\frac{L^{-\alpha}}{\tr (L^{-\alpha})}\,,
\end{align}
where $a_{o}$ is given in \eqref{optimal.estimator}. For $\alpha =1$, this is the estimator considered by \cite{Konno2009}, who deals with the Gaussian case and the quadratic loss \eqref{quadratic.loss}, while  \cite{tsukuma2020estimation} used an extended Stein loss. An elliptical setting was also considered by \cite{HADDOUCHE2021104680} under  the quadratic loss \eqref{quadratic.loss}.

It is proved in the Appendix that, for the entire class of elliptically symmetric distributions in \eqref{density},   any estimator  $\hat{\Sigma}_{\alpha,b}$ in \eqref{haff.estimators} improves on the optimal estimator $\hat{\Sigma}_{a_o}$ in \eqref{optimal.estimator}, under the data--based loss \eqref{loss}, as soon as
\begin{align}\label{improvement.c}
	0 < b \leq \frac{2\,(r -1)}{v - r +1 }\,.
\end{align}
It worth noting that  \cite{tsukuma2020estimation}  gave  Condition  \eqref{improvement.c} as an improvement condition although their loss was different.
\section{Numerical study}\label{numerical.study}

Let the elliptical density in \eqref{noise.density} be a variance mixture of normal distributions where the mixing variable, with density $h$, has the inverse--gamma distribution  ${\cal IG}(k/2,k/2)$  with shape and scale parameters both equal to $k/2$ for $k>2$.
Thus, for any $t \geq 0$, the generating function $f$ in \eqref{noise.density} has the form 
\begin{align*}
	f(t)&= \int_{0}^{\infty} \frac{1}{(2{\tt v}\pi)^{np/2}}\exp\left( \frac{-t}{2{\tt v}}\right)\,h({\tt v})\,d {\tt v}\,,
\end{align*} 
which corresponds to the $t$--distribution with $k$ degrees of freedom. Then the primitive $F^{*}$ of $f$ in \eqref{density.F*} is, for any $t\geq 0$,
\begin{align*}
	F^{*}(t)
	&
	=\frac{1}{2}\int_{t}^{\infty}\int_{0}^{\infty} \frac{1}{(2{\tt v}\pi)^{np/2}}\exp\left( \frac{-w}{2{\tt v}}\right)\,h({\tt v})\,d{\tt v}\,d{w}  
	%
	=\int_{0}^{\infty} \frac{{\tt v}}{(2{\tt v}\pi)^{np/2}}\exp\left(\frac{-t}{2{\tt v}}\right)\,h({\tt v})\,d{\tt v}\,.
\end{align*}
by Fubini's theorem. Therefore the normalizing constant $K^{*}$ in  \eqref{density.F*}  is 
\begin{align}\label{K*.student}
	K^{*}
	&=
	\int_{\R^{pn}}\int_{0}^{\infty} \frac{|\Sigma|^{-n/2}}{(2{\tt v}\pi)^{np/2}}\,{\tt v}\,\exp\left(\frac{-1}{2{\tt v}}\,\big(\,\tr\, ({z} 
	- {\theta})\,\Sigma^{-1}\,({z}-
	{\theta})^{\T } + \tr\, \Sigma^{-1}\,{u}^{\!\T }
	{u}\,\big)\right)h({\tt v})\,d{\tt v}\,\,dz\,du\,, 
	\nonumber
	\\
	&=\int_{0}^{\infty}{\tt v}\int_{\R^{pn}} \frac{|\Sigma|^{-n/2}}{(2\,{\tt v}\pi)^{np/2}}\,\exp\left(\frac{-1}{2{\tt v}}\,\big(\,\tr\, ({z} 
	- {\theta})\,\Sigma^{-1}\,({z}-
	{\theta})^{\T } + \tr\, \Sigma^{-1}\,{u}^{\!\T }
	{u}\,\big)\right)\,dz\,du\,h({\tt v})\,d{\tt v}\, 
\end{align}
by Fubini's theorem. Clearly the most inner integral in \eqref{K*.student} equals 1 so that

\begin{align*}
K^{*}&=\int_{0}^{\infty}{\tt v}\,h({\tt v})\,d{\tt v}
=\frac{k}{k-2},
\end{align*}
by propriety of ${\cal IG}(k/2,k/2)$. Note that, when $k$ goes to $\infty $,  ${\cal IG}(k/2,k/2)$ goes to the multivariate Gaussian distribution  (for which $K^{*}=1$ since $f=F^*$) with covariance matrix  $I_n \otimes \Sigma$.

In the following, we study numerically the performance of the  alternative estimators in \eqref{haff.estimators} expressed as 
\begin{align}\label{haff.estimators.bis}
	\hat{\Sigma}_{\alpha,b} = a_{o}\left(S + \frac{b}{\tr (L^{-\alpha})}H\,L^{1-\alpha}\,H^{\T}\right) \quad \text{where} \quad 0\leq b\leq b_{0}=\frac{2\,(r -1)}{v - r +1 }\quad \text{and} \quad \alpha \geq 1.
\end{align}
%
%
As mentioned above,  \cite{Konno2009} consider the case  $\alpha=1$, in the Gaussian setting and under the quadratic loss (\ref{quadratic.loss}), for which its improvement condition is
$$0 \leq b \leq b_1=\frac{2\,(r -1)\,(v+r+1)}{(v - r +1)\,(v-r+3)}.$$
%
Note that, although $b_{0}< b_{1}$, the improvement condition in \eqref{haff.estimators.bis} is valid fo any $\alpha \geq 1$ and all the class of elliptically symmetric distributions \eqref{density}. However 
it was  shown numerically by \cite{HADDOUCHE2021104680} that $b_1$  is optimal in the Gaussian context.

We consider the following structures  of $\Sigma$: $\rm(i)$ the identity matrix $I_p$ and $\rm(ii)$ an autoregressive structure with coefficient $0.9$ (i.e. a $p\times p$ matrix where the $(i,j)$th element  is $0.9^{|i-j|}$). To assess how an alternative estimator $\hat{\Sigma}_{\alpha,b}$ improves over  $\hat{ \Sigma}_{a_o}$, we compute the Percentage Reduction In Average Loss (PRIAL) 
defined as 
\begin{align*}
	{\rm PRIAL}(\hat{\Sigma}_{\alpha,b})=\frac{ {\rm average\,\, loss\,\,of\,\,} \hat{ \Sigma}_{a_o} - {\rm average\,\, loss\,\,of\,\,} \hat{ \Sigma}_{\alpha,b}}{{\rm average\,\, loss\,\,of\,\,} \hat{ \Sigma}_{a_o}}\,
\end{align*}
and based on $1000$ independent Monte--Carlo replications for some couples $(p,m)$.

In Figure \ref{fig0}, we study the effect of the constant $b$ in \eqref{haff.estimators.bis} on the prial's in the non--invertible  ($(p,m)=(25,10)$) and the invertible  ($(p,m)=(10,25)$) cases. The  Gaussian setting is investigated for the structure $\rm{(i)}$ of $\Sigma$. Note that, when $0\leq b \leq b_{0} $, the best prial (around $7\%$ in both invertible and non--invertible cases) is reported for $b=b_{0}=1.125$ (for $(v,r)=(25,20)$). For this reason, in the following,  we consider the  estimators  $\hat{ \Sigma}_{\alpha,b_0}$ with  
$$b_{0}=\frac{2\,(r -1)}{v - r +1 }\,.$$
Note also that,  for $b> b_0$, the  estimators $\hat{ \Sigma}_{\alpha,b}$ still improve over $\Sigma_{a_{o}}$ and that the maximum value of the prial is around $50\%$.
This shows that there exists a larger range of values of $b$ than the one our theory provides for which $\hat{ \Sigma}_{\alpha,b}$ improves over $\hat\Sigma_{a_{o}}$.

In Figure \ref{fig1}, we study the effect of $\alpha$ on the prial's of the  estimator $\hat{\Sigma}_{\alpha,b_0}$ over $\hat{ \Sigma}_{a_o}=S/v$ when the sampling distribution is Gaussian ($K^{*}=1$ in \eqref{optimal.estimator}),
%
and over $\hat{ \Sigma}_{a_o}=S(k-2)/vk $ when it is the $t$-distribution ($K^{*}=(k-2)/k$ in \eqref{optimal.estimator}) with $k$ degrees of freedom. 
For the structure $\rm(i)$ of $\Sigma$, note that, for $\alpha\geq 6$, the prial's stabilize at $12.5\%$, in the Gaussian case, and at  $8.5\%$, in the Student case. Similarly, the prial's are better in the Gaussian setting for the structure $\rm(ii)$. In addition, it is interesting to observe that, when $\alpha$ is close to zero, the prial's are small for the structure $\rm(i)$ and may be negative for the structure $\rm(ii)$.

In Figure \ref{fig2}, under the Gaussian assumption,  we provide the prial's of $\hat{ \Sigma}_{\alpha,b_0}$  with respect to $\hat{ \Sigma}_{a_o}=S/v$ under the data--based loss \eqref{loss} and  the prial's of $\hat{ \Sigma}_{\alpha,b_1}$ with respect to $\hat{ \Sigma}_{a_o}=S/(v+r+1)$ under the quadratic loss \eqref{quadratic.loss}. For the two structures $\rm{(i)}$ and $\rm{(ii)}$ of $\Sigma$,  the prial's are better under the data--based loss. For the structure $\rm{(i)}$ with $\alpha=1$ (which coincide with the Konno's estimator), we observe a prial equal to $1.73\%$ which is similar to that of  \cite{Konno2009}. Note that, under the data--based loss the prial is much better since it equals  $13.42\%$. We observe similar behaviors for the structure $\rm{(ii)}$ than  for the structure $\rm{(i)}$, but with lower prial's.
\begin{figure}[p]
		\centering
		\includegraphics[width=.4\linewidth]{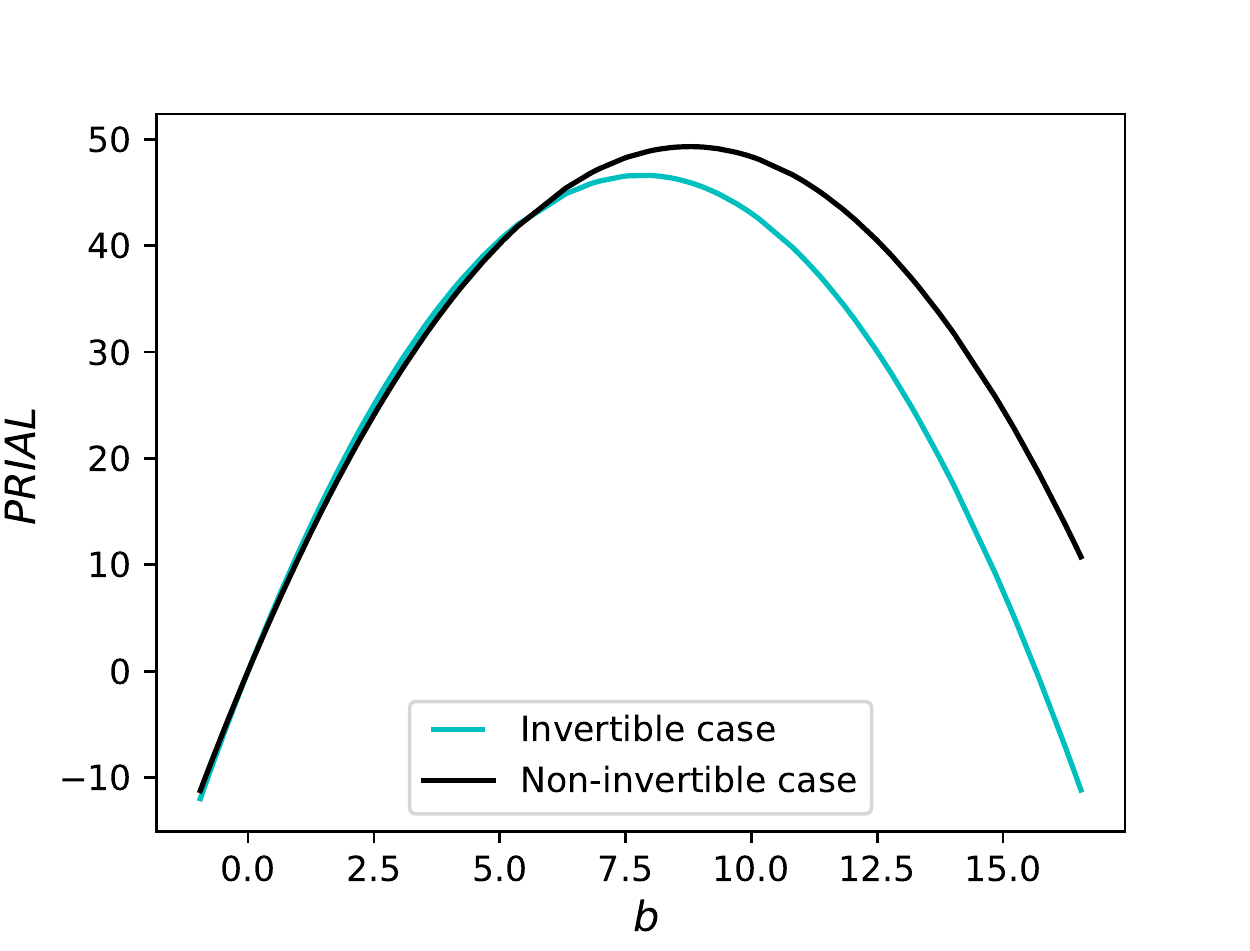}	
	\caption{Effect of $b$ on the PRIAL of $\hat{\Sigma}_{\alpha,b}$, with $\alpha=1$, under data--based loss in the Gaussian setting. The structure  ${\rm (i)}$ of $\Sigma$ is considered for the invertible case with $(p,m)=(10,25)$ and the non--invertible case with $(p,m)=(25,10)$.}
	\label{fig0}
\end{figure}
\begin{figure}[p]
	\begin{subfigure}{.5\textwidth}
		\centering
		\includegraphics[width=.75\linewidth]{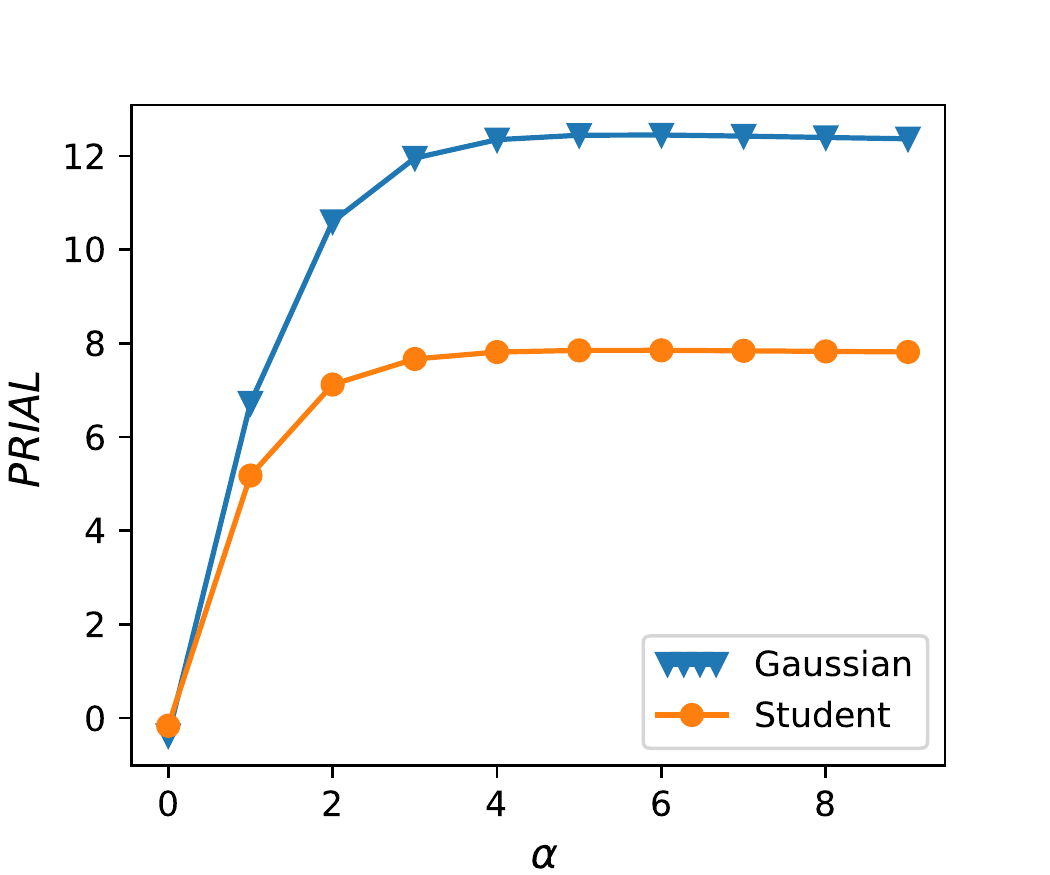}
		\caption{}
	\end{subfigure}%
	\begin{subfigure}{.5\textwidth}
		\centering
		\includegraphics[width=.75\linewidth]{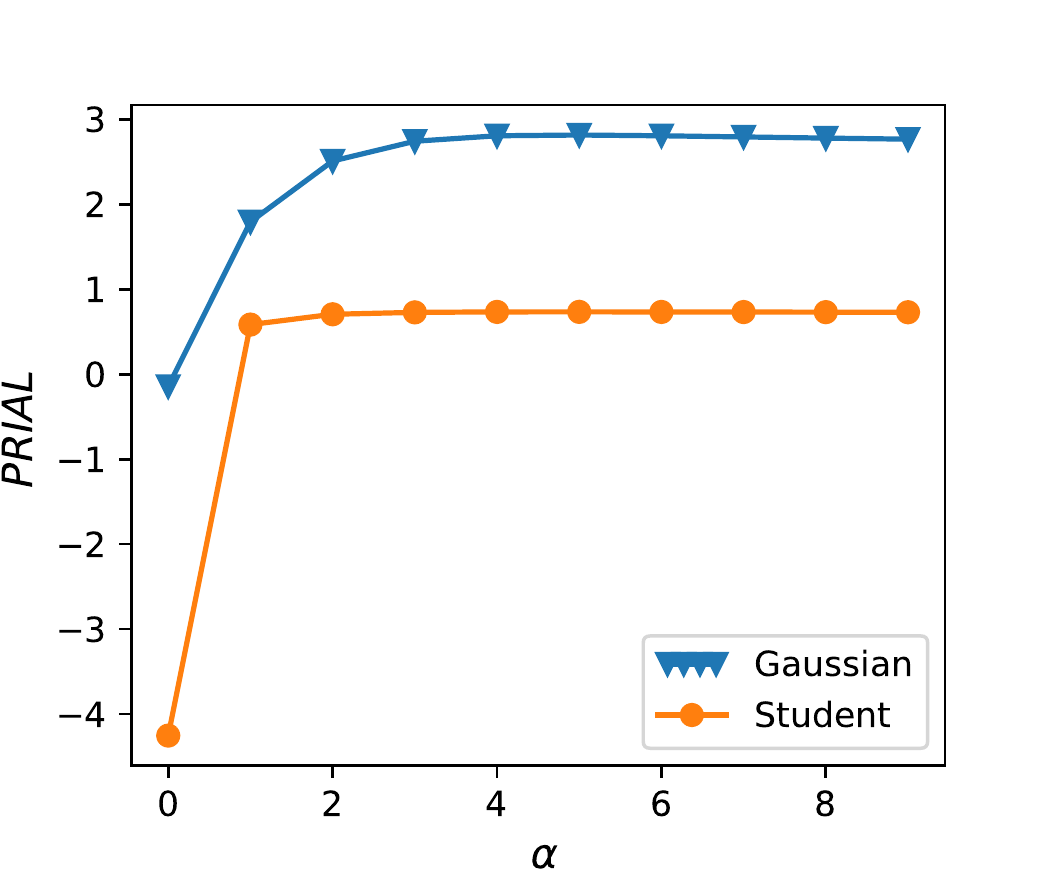}
		\caption{}
	\end{subfigure}
	
	\caption{PRIAL's of $\hat{\Sigma}_{\alpha,b_0}$ under the data--based loss. The non-invertible case  is considered, with $(p,m)=(50,20)$, for the structures ${\rm (i)}$ and ${\rm (ii)}$  of $\Sigma$ for the t-distribution, with $k=5$ degrees of freedom, and the Gaussian distribution. }
	\label{fig1}
\end{figure}
\begin{figure}[p]
	\begin{subfigure}{.50\textwidth}
		\centering
		\includegraphics[width=.75\linewidth]{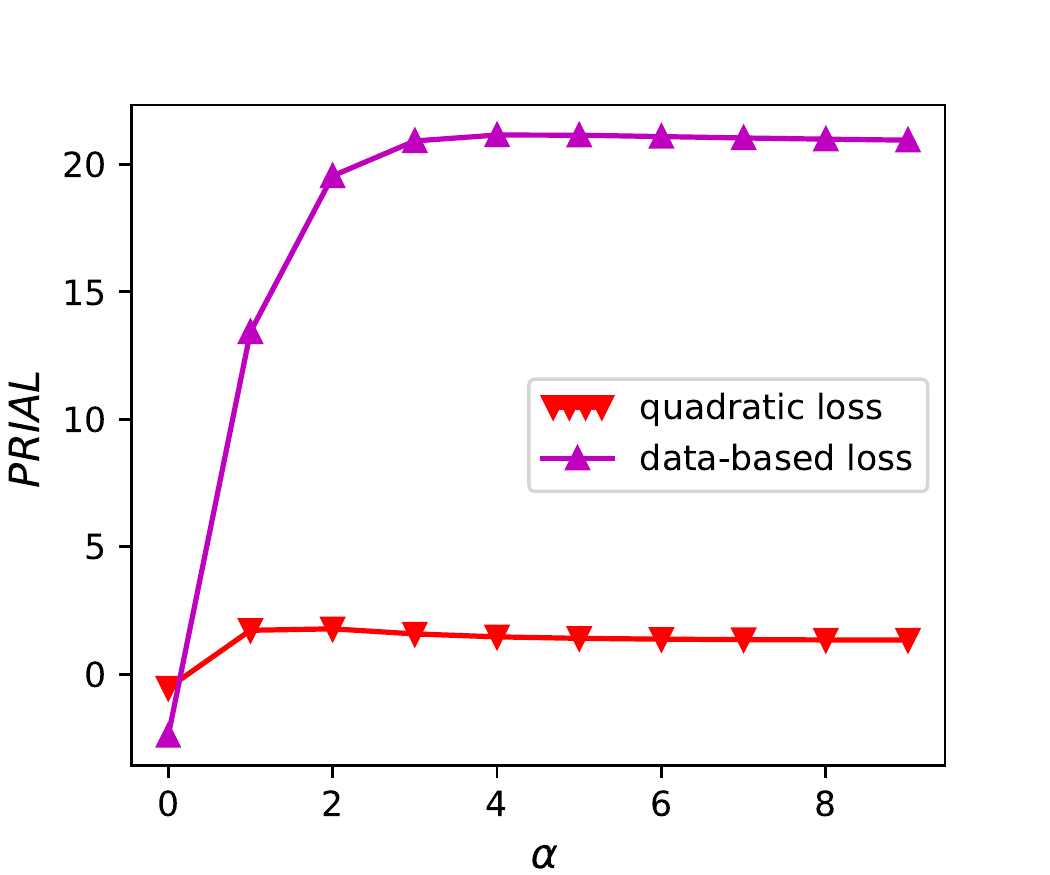}
		\caption{}
	\end{subfigure}%
	\begin{subfigure}{.50\textwidth}
		\centering
		\includegraphics[width=.75\linewidth]{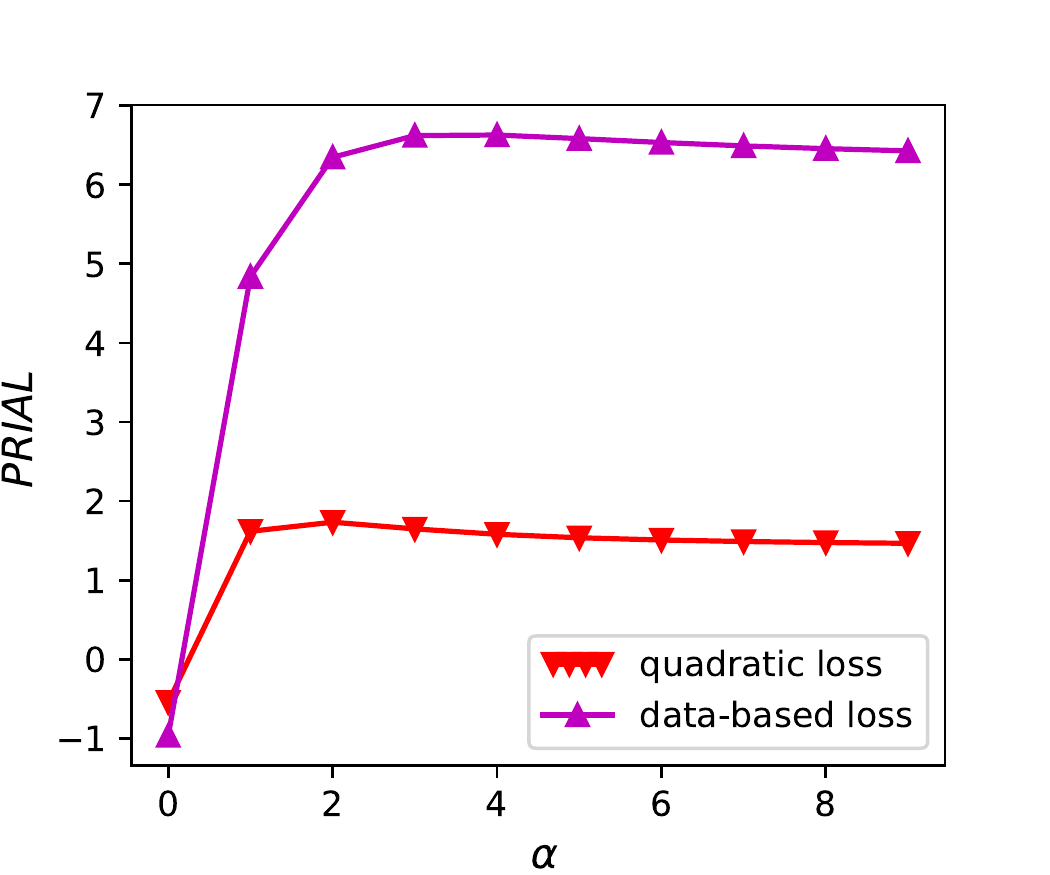}
		\caption{}
	\end{subfigure}
	\caption{PRIAL's of $\hat{\Sigma}_{\alpha,b_0}$ under data--based loss and PRIAL's of  $\hat{\Sigma}_{\alpha,b_1}$ under quadratic loss. The non--invertible case  is considered, with $(p,m)=(20,10)$, for the structures ${\rm (i)}$ and ${\rm (ii)}$  of $\Sigma$ under the Gaussian distribution.}
	\label{fig2}
\end{figure}

\newpage
\section{Conclusion and perspective}
For a wide class of elliptically symmetric distributions, we provide a large class of estimators of the scale matrix $\Sigma$ of the elliptical multivariate linear model \eqref{linear.model} which improve over the usual estimators $a\,S$. We highlight that the use of the data--based loss \eqref{loss} is more attractive than the use of the classical quadratic loss \eqref{quadratic.loss}. Indeed, \eqref{loss}  brings more improved estimators and their improvement is valid within a larger class of distributions. This means that \eqref{loss} is more discriminant than \eqref{quadratic.loss} to exhibit improved estimators.

%
%
%
%

While in \eqref{Delta.SG.2.S}  the risk difference between  $\hat\Sigma_{J}=a_{o}(S+J)$ with $J=SS^{+}G(Z,S)$ and $\hat{ \Sigma}_{a_o}=a_o\,S$, the dominance result in Theorem \ref{risk.diff.} is given for a correction matrix $G(Z,S)=HL\Psi(L)H^{\T}$ which depends only on $S$. Recently,  \cite{tsukuma2016} consider, in the Gaussian case, alternative estimators where $G(Z,S)$ depends on $S$ and on the information contained in the sample mean $Z$. This class of estimators merits future investigations in an elliptical setting. 


\section{Appendix}
\numberwithin{equation}{section}
\numberwithin{lemma}{section}
\numberwithin{proposition}{section}
\numberwithin{corollary}{section}

%
We give in the following corollary  an adaptation of Lemma \eqref{Stein-Haff.id} to an orthogonally invariant matrix function $G$, that is, of the form $G=H\,L\,\Phi(L)\,H^{\T}$ where $\Phi(L)={\rm diag}(\phi_1,\dots,\phi_r)$ with $\phi_i=\phi_i(L)$ ($i=1,\dots,r$) is differentiable function of $L$

\begin{corollary}\label{steinhaff.oiv}
	Let $\Phi(L)={\rm diag}(\phi_1,\dots,\phi_r)$ where $\phi_i=\phi_i(L)$ ($i=1,\dots,r$) is differentiable function of $L$. Assume that $\Exp{} \big[| \tr (\Sigma^{-1} H\,L\,\Phi(L)\,H^{\T}) |\big ] < \infty$. Then we have
	\begin{align*}
		\Exp{}\big[\tr (\Sigma^{-1}\,H L\,\Phi(L)\,H^{\T}\,)\big] 
		=K^{*}\,\Exp{*}\left[\sum_{i=1}^{r} \big(	(v-r +1 )\,\phi_i + 2\,l_i\,\frac{\partial \phi_i}{\partial l_i} 
		+\sum_{j\neq i}^{r}\frac{l_i\,\phi_i-l_j\,\phi_j}{l_i -l_j}\,\big) \right]\,.
	\end{align*}
\end{corollary}
\begin{proof}[\bf Proof]
Let $G=H\,L\,\Phi(L)\,H^{\T}$, $S^{+}=H\,L^{-1}\,H^{\T}$ and $SS^{+}=H\,H^{\T}$. Then,
$$SS^{+}G=H\,H^{\T}H\,L\,\Phi(L)\,H^{\T}=H\,L\,\Phi(L)\,H^{\T}= G,$$ since $H$ is semi--orthogonal.
Assuming that $\Exp{} \big[| \tr (\Sigma^{-1} H\,L\,\Phi(L)\,H^{\T}) | \big] < \infty$, we have from Lemma \ref{lemma.id.hs}  
\begin{align}\label{tool.c1.1}
\Exp{}\big[\tr \big(\Sigma^{-1}\,H L\,\Phi(L)\,H^{\T}\,\big)\big] 
=
K^{*}\Exp{*} \big[2\,\tr\big(H\,H^{\T}{\cal D}_s\{H\,L\,\Phi(L)\,H^{\T}\}\big) + (m-r-1)\,\tr\big(	H\,\Phi(L)\,H^{\T} \big)\big	]\,.
\end{align}

Firstly, using Lemma A.4.2 in \cite{HADDOUCHE2021104680}, we have 
\begin{align}
		&{\cal D}_s \big\{H\,L\,\Phi(L)\, H^{\T}\big\} 
		= H \Phi^{(1)}(L)H^{\T } + \frac{1}{2} \tr\big(\Phi(L)\big)\big(I_p - HH^{\T }\big)\,,\label{tool.c1.3} \\ 
		\intertext{where $\Phi^{(1)}(L)={\rm diag}(\phi^{(1)}_1,\dots,\phi^{(1)}_r)$, with}
		&\phi^{(1)}_i
		= \frac{1}{2}(p- r +2 )\,\phi_i + l_i\,\frac{\partial \phi_i}{\partial l_i} 
		+\frac{1}{2}\sum_{j\neq i}^{r}\frac{l_i\,\phi_i-l_j\,\phi_j}{l_i -l_j}\,.
		\label{tool.c1.4}
\end{align}
for $i=1\dots r$. 

Secondly, using the fact that $H^{\T}H=I_r$, we have from \eqref{tool.c1.3} 
\begin{align}\label{tool.c1.2}
		H\,H^{\T}{\cal D}_s \big\{H\,L\,\Phi(L)\, H^{\T}\big\} 
		=H\,\Phi^{(1)}(L)\,H^{\T\,}\,.
\end{align}
Then, putting \eqref{tool.c1.2} in \eqref{tool.c1.1}, we obtain 
\begin{align*}
		\Exp{}\big[\tr \big(\Sigma^{-1}\,H L\,\Phi(L)\,H^{\T}\,\big)\big] 
		=
		K^{*}\Exp{*} \big[2\,\tr\big(\Phi^{(1)}(L)\big) + (m-r-1)\,\tr\big(\Phi(L) \big)\big]\,.
\end{align*}

Finally, using \eqref{tool.c1.4}, we have 
	\begin{align*}
		\Exp{}\big[\tr \big(\Sigma^{-1}\,H L\,\Phi(L)\,H^{\T}\,\big)\big] 
		=K^{*}\,\Exp{*}\left[\sum_{i=1}^{r} \big(	(p+m -2r +1 )\,\phi_i + 2\,l_i\,\frac{\partial \phi_i}{\partial l_i} 
		+\sum_{j\neq i}^{r}\frac{l_i\,\phi_i-l_j\,\phi_j}{l_i -l_j}\,\big) \right]\,,
	\end{align*}
	where $(p+m -2r +1 ) = (v-r)$ \,.
\end{proof}

\begin{proof}[\bf The optimal constat $a_{o}$ in \eqref{optimal.estimator}]
Let $\hat{\Sigma}_{a}= a\,S$ where $a>0$. Assume that the expectations $\Exp{}\big[\tr\big(\Sigma^{-1}S\big)\big]$ and $\Exp{}\big[\tr\big(\Sigma\, S^{+}\big)\big]$ are finite. Then, the risk of $\hat{\Sigma}_{a_{o}}$ relating to the data-based loss \eqref{loss} is given by
	\begin{align}
		R\big(\hat{\Sigma}_a,\Sigma)
		= 
		\Exp{}\big[\tr\big(S^{+}\Sigma\,(\Sigma^{-1}\hat{\Sigma}_a - I_p  )^2	\big)\big]
		=
		a^{2}\Exp{}\big[\tr\big(\Sigma^{-1}SS^{+}S\big)\big] -2\,a\,\Exp{}\big[ \tr \big(SS^{+}\big)	\big]+ \Exp{}\big[ \tr\big(S^{+}\,\Sigma \big)\big] \label{tool.p1.1}\,.
	\end{align}
	Applying the Stein-Haff type identity in Corollary \eqref{steinhaff.oiv}, with $\Psi(L)=I_r$, to the first term in the right-hand side of \eqref{tool.p1.1}, we obtain
	\begin{align} \label{tool.p1.2}
		\Exp{}\big[\tr\big(\Sigma^{-1}SS^{+}S\big)\big]
		&= \Exp{}\big[\tr\big(\Sigma^{-1}H\,L\,H^{\top}\big)\big]
		=K^{*}\,\Exp{*}\left[\sum_{i=1}^{r} \big((v-r +1 )\, 
		+\sum_{j\neq i}^{r}\frac{l_i -l_j\,}{l_i -l_j}\,\big) \right]\, \nonumber \\
		&=K^{*}\,\left[r(v-r+1) + r(r-1)\right]
		=K^{*}r\,v\,.
	\end{align}
	Now, using the fact that $\tr(S^{+}\,S)=\tr(H\,H^{\T})= r$ and thanks to \eqref{tool.p1.2}, we have 
	\begin{align*}
		R\big(\hat{\Sigma}_a,\Sigma\big)=a^2\,K^{*}\,r\,v - 2\,a\,r + \Exp{}\big[	\tr\big(S^{+}\Sigma \big) \big]\,.
	\end{align*}
	Therefore, choosing $a= 1/K^{*}\,v$ is optimal under the risk \eqref{risk}.
\end{proof}
\begin{proof}[\bf Proof of Theorem \ref{risk.diff.}]
	Let $\hat{\Sigma}_\Psi= a_{o}\,\big(S + H\,L\,\Psi(L)\,H^{\T})$ where  $\Psi(L)={\rm diag}(\psi_1,\dots,\psi_r)$ such that $\psi_i=\psi_i(L)$ ($i=1,\dots,r$) is differentiable function of $L$ and  $\tr\big(\Psi(L)\big)\geq \lambda >0$. 
	Hence, using the fact that $H^{\top}H=I_r$, the involving terms in the risk difference \eqref{Delta.SG.1} becomes
	\begin{align*}
		J=SS^{+}G=G=HL\Psi(L)H^{\top}
		\,\,
		\text{and}
		\,\,
		S^{+}G= H\Psi(L)H^{\top}\,.
		\end{align*} 
	Then, the risk difference between  $\hat{\Sigma}_\Psi$  and $\hat{\Sigma}_{a_o}$ is given by
	\begin{align}\label{tool.t21.2}
		\Delta(\Psi)= 
		a_{o}^{2}\,\Exp{}\big[\tr\big(\Sigma^{-1}\,H\,L\,(2\,\Psi +\,\Psi^2 )\,H^{\T}\big	)  \big] - 2\,a_{o}\Exp{}\big[\tr\big(\Psi\big)\big]\,.
	\end{align}
	
	Now, applying the Stein-Haff type identity in Corollary \eqref{steinhaff.oiv} to the first term in the right hand side of \eqref{tool.t21.2}, for $\Phi = 2\,\Psi + \Psi^{2}$, we have
	\begin{align*}
		\Delta(\Psi)
		&=
		a_{o}^2\,K^{*}\,\Exp{*}\left[\sum_{i=1}^{r} \big\{	(v-r +1 )\,(2\,\psi_i + \psi_i^2)
		+ 2\,l_i\,\frac{\partial (2\,\psi_i + \psi_i^2)}{\partial l_i} 
		+\sum_{j\neq i}^{r}\frac{l_i\,(2\,\psi_i + \psi_i^2)-l_j\,(2\,\psi_j + \psi_i^2)}{l_i -l_j}\,\big \} \right] 
		\nonumber
		\\
		&\hspace{2.cm} 
		- 2\,a_{0}\Exp{}\big[\tr\big(\Psi\big)\big]\,.
	\end{align*}
	Therefore, using the fact that $\tr(\Psi)\geq \lambda>0$, an upper bound of the risk difference $\Delta(\Psi)$ is given by
	\begin{align*}
		\Delta(\Psi) 
		&\leq
		a_{o}^2\,K^{*}\,\Exp{*}\left[\sum_{i=1}^{r} \left\{ 2\,(v-r+1)\,\psi_i
		+ (v-r+1)\,\psi_i^2 
		+4\,l_i\,(1+\psi_i )\frac{\partial \psi_i}{\partial l_i}
		\right.\right. \nonumber	\\
		&\hspace{2.cm}
		+\sum_{j\neq i}^{r}\frac{l_i\,(2\,\psi_i + \psi_i^2)-l_j\,(2\,\psi_j + \psi_i^2)}{l_i -l_j}\,
		\left.\left.
		-2(a_o\,K^{*})^{-1}\lambda \right \}\right]\,,
	\end{align*} 
where  $(a_o\,K^{*})^{-1} = v $.
\end{proof}
\begin{proof}[\bf Improvement condition \eqref{improvement.c} of  alternative estimators  in \eqref{haff.estimators}]
	Let consider the class of alternative estimators $\hat{\Sigma}_{\alpha,b}$  in \eqref{haff.estimators}. Then, applying Theorem  \ref{risk.diff.}, an upper bound of the risk difference between $\hat{\Sigma}_{\alpha,b}$  and $\hat{\Sigma}_{a_o}$ is given by
	\begin{align}\label{hf.risk}		
		\Delta(\Psi) \leq a_{o}^{2}\,K^{*}\,\Exp{*}\big( g (\Psi)\big)\,,
	\end{align}
	where the integrand term in \eqref{tool.t12.3} becomes
	\begin{align*}
		g(\Psi)= g_1(\Psi) + g_2(\Psi)
	\end{align*}
	with
	\begin{align*}
		g_1(\Psi) = -2\,(r -1 )\,b\,\sum_{i=1}^{r} \frac{\,l^{-\alpha}_i}{\tr(L^{-\alpha})} 
		+ (v - r +1 )\,b^{2}\,\sum_{i=1}^{r}\frac{\,l^{-2\alpha}_i}{\tr^{2}(L^{-\alpha})}\,,
	\end{align*}
	since $\tr\big(\Psi(L)\big)=b$, and
	\begin{align*}
		&
		g_2(\Psi)
		=4l_i b\left(1+b\frac{l^{-\alpha}_i}{\tr(L^{-\alpha})}\right)\frac{\partial}{\partial l_{i}} \left(\frac{l^{-\alpha}_i}{\tr(L^{-\alpha})} \right) 
		+ \frac{2b}{\tr(L^{-\alpha})}\sum_{i=1}^{r} \sum_{j\neq i}^{r} \frac{l_i^{1-\alpha} - l_j^{1-\alpha}}{l_i -l_j} 
		\nonumber \\
		&\hspace{8cm}
		+ \frac{b^2}{\tr^{2}(L^{-\alpha})}\sum_{i=1}^{r} \sum_{j\neq i}^{r} \frac{l_i^{1-2\alpha}- l_j^{1-2\alpha}}{l_i -l_j}\,.
	\end{align*}
	The proof consist to prove that the integrand term $g_2(\Psi)$ is non-positive. To this end, it can be shown that, for $\alpha \geq 1$,
	\begin{align*}
		\sum_{i=1}^{r} \sum_{j\neq i}^{r} \frac{l_i^{1-\alpha} - l_j^{1-\alpha}}{l_i -l_j} 
		= 2\,\sum_{i}^{r} \sum_{j> i}^{r} \frac{l_i^{1-\alpha} - l_j^{1-\alpha}}{l_i -l_j} \leq 0
		\quad{\text{and}}\quad
		\sum_{i=1}^{r} \sum_{j\neq i}^{r} \frac{l_i^{1-2\alpha}- l_j^{1-2\alpha}}{l_i -l_j}
		=  2\,\sum_{i=1}^{r} \sum_{j> i}^{r} \frac{l_i^{1-2\alpha}- l_j^{1-2\alpha}}{l_i -l_j} <0\,.
	\end{align*}
	since $L={\rm diag}(l_1>,\dots,>l_r)$. Then 
	\begin{align*}
		g_2(\Psi) 
		&\leq
		4\,l_i\,b\,\left(1+b\,\frac{\,l^{-\alpha}_i}{\tr(L^{-\alpha})}\right)\,\frac{\partial}{\partial l_{i}} \left(\frac{\,l^{-\alpha}_i}{\tr(L^{-\alpha})} \right) 
		=4\,b\,\alpha\,\frac{l_i^{-\alpha}}{\tr(L^{-\alpha})}\left(1+b\,\frac{\,l^{-\alpha}_i}{\tr(L^{-\alpha})}\right) \,\left(\frac{l_i^{-\alpha}}{\tr(L^{-\alpha})}-1\right)\,,
	\end{align*}
	since 
	\begin{align*}
		\frac{\partial}{\partial l_{i}}\left(\frac{\,l^{-\alpha}_i}{\tr(L^{-\alpha})} \right)
	=\alpha \frac{l_i^{-\alpha-1}}{\tr(L^{-\alpha})}\left (\frac{l_i^{\,-\alpha}}{\tr(L^{-\alpha})} - 1 \right) \,.
	\end{align*}
	Therefore, since $l_i^{-\alpha}\leq \tr(L^{-\alpha})$, the integrand term $g_2(\Psi) \leq0$. Then
	\begin{align*}
		g(\Psi)\leq g_1(\Psi)=  -2\,(r -1 )\,b\,
		+ (v - r +1 )\,b^{2}\,\frac{\tr(L^{-2\alpha})}{\tr^{2}(L^{-\alpha})}\,.
	\end{align*}
	Now, using the fact that $\tr(L^{-2\alpha})\leq \tr^{2}(L^{-\alpha})$, we have 
	\begin{align*}
		g(\Psi)
		\leq
		-2\,(r -1 )\,b
		+ (v - r +1 )\,b^{2}\,.
	\end{align*}
	since $b>0$. Hence, an upper bound for the risk difference  in \eqref{hf.risk} is given by 
	\begin{align*}
		\Delta(\Psi)\leq a_{o}^{2}\,b\,K^{*}\,\Exp{*}\big[ -2\,(r -1 )\,
		+ (v - r +1 )\,b\big]\,.
	\end{align*}
	Therefore,  $\hat{\Sigma}_{\alpha,b}$  improves over $\hat{\Sigma}_{a_o}$ under the data-based loss \eqref{loss} as soon as   $0 < b \leq b_0= 2\,(r -1 ) /(v - r +1 )$\,.
\end{proof}

\printcredits
\newpage
\bibliographystyle{cas-model2-names}

\bibliography{FMH2020_DB_biblio}

\end{document}